\documentclass[11pt,leqno]{amsart}

\usepackage[top=2.5 cm,bottom=2 cm,left=2.5 cm,right=2 cm]{geometry}
\usepackage[utf8]{inputenc}
  \usepackage{color}

\usepackage{esint}
\usepackage{graphicx}
\usepackage{hyperref}

\newtheorem{definition}{Definition}[section]
\newtheorem{theorem}{Theorem}[section]
 
\newtheorem{lemma}{Lemma}[section]
 \newtheorem{corollary}{Corollary}[section]
\newtheorem{remark}{Remark}[section]
\newtheorem*{maintheorem*}{Main Theorem}
\allowdisplaybreaks
\numberwithin{equation}{section}

\newcommand{\lsim}{\raisebox{-0.13cm}{~\shortstack{$<$ \\[-0.07cm]
      $\sim$}}~}

\renewcommand{\i}{\ifmmode\mathit{\mathchar"7010 }\else\char"10 \fi}
\renewcommand{\j}{\ifmmode\mathit{\mathchar"7011 }\else\char"11 \fi}
\newcommand{\R}{\mathbb{R}}
\newcommand{\N}{\mathbb{N}}

\newcommand{\abs}[1]{\left|#1\right|}

\newcommand{\pt}{\partial_t}

\newcommand{\ptt}{\partial_{tt}^2}

\newcommand{\px}{\partial_x }
\newcommand{\pxx}{\partial_{xx}^2}
\newcommand{\pxxx}{\partial_{xxx}^3}
\newcommand{\pxxxx}{\partial_{xxxx}^4}

\newcommand{\vfi}{\varphi}

\def\begi{\begin{itemize}}
\def\endi{\end{itemize}}
\def\bega{\begin{array}}
\def\enda{\end{array}}

{%

\begin{enumerate}}%
{\end{enumerate}}

{%

\begin{enumerate}}%
{\end{enumerate}}

\begin{document}

\title[Regularity and energy transfer for a nonlinear beam equation]{Regularity and energy transfer for a nonlinear beam equation}

\author{G. M. Coclite}
\address[Giuseppe Maria Coclite]{\newline
   Dipartimento di Meccanica, Matematica e Management, Politecnico di Bari,
  Via E.~Orabona 4,I--70125 Bari, Italy.}
\email[]{giuseppemaria.coclite@poliba.it}

\author{G. Fanizza}
\address[Giuseppe Fanizza]{\newline
  Instituto de Astrofis\'\i ca e Ci\^encias do Espa\c co, Faculdade de Ci\^encias, Universidade de Lisboa,
Edificio C8, Campo Grande, P-1740-016, Lisbon, Portugal.}
\email[]{gfanizza@fc.ul.pt}

\author{F. Maddalena}
\address[Francesco Maddalena]{\newline
  Dipartimento di Meccanica, Matematica e Management, Politecnico di Bari,
  Via E.~Orabona 4,I--70125 Bari, Italy.}
\email[]{francesco.maddalena@poliba.it}

\date{\today}



\thanks{GMC and FM are members of the Gruppo Nazionale per l'Analisi Matematica, la Probabilit\`a e le loro Applicazioni (GNAMPA) of the Istituto Nazionale di Alta Matematica (INdAM). 
GMC  and FM have been partially supported by the  Research Project of National Relevance ``Multiscale Innovative Materials and Structures'' granted by the Italian Ministry of Education, University and Research (MIUR Prin 2017, project code 2017J4EAYB and the Italian Ministry of Education, University and Research under the Programme Department of Excellence Legge 232/2016 (Grant No. CUP - D94I18000260001). GF acknowledges support by FCT under the program {\it Stimulus} with the grant no. CEECIND/04399/2017/CP1387/CT0026.}

\begin{abstract} 
In this paper we study some key effects of a discontinuous forcing term in a fourth order wave equation on a bounded domain, 
modeling  the adhesion of an elastic beam with a substrate through an elastic-breakable interaction.
 By using a spectral decomposition method we show that the main effects induced by the nonlinearity  at the transition from attached to detached states can be traced in a loss of regularity of the solution and in a migration of the total energy through the scales.
\end{abstract}

\maketitle

\section{Introduction}
\label{sec:intro}
The essential mechanism underlying the manifestation of adhesion phenomena in nature relies on the possible intermittency between the two states in which the material bodies experience partial contact or separation and an interesting mathematical problem consists in understanding the effects of this occurrence on the dynamical evolution problem. More precisely, the tricky question asks  for a precise quantification of the nonlinearity  injected    into the system through a discontinuous forcing term.

In this paper we deal with the initial boundary value problem ruled by the elastic  beam equation complemented by a source term
undergoing a sharp discontinuity when a threshold on the displacement is reached. The analysis performed in this note suggests that the transition between the attached-detached states entails two main effects, namely the loss of regularity in the velocity field and a migration of energy through the  scales. We believe that this result suggests an interesting perspective (as in \cite{MRTT}) in the dynamical analysis of a large class of problems as the ones studied, for instance,  in \cite{BK, MPPT}.

The evolution of an elastic beam according to the Bernoulli-Navier model interacting with a rigid substrate through an elastic-breakable forcing term leads to the following   semilinear initial boundary value problem in the unknown function $u(t,x)$ denoting the displacement at the time $t\ge0$ of the material point located in the reference configuration at $x\in [0,L]$:
\begin{equation}
\label{eq:equation}
\begin{cases}
\ptt u=-\kappa^2_1\pxxxx u-\Phi'\left(u\right)&\quad t>0,0<x<L,\\
\px u(t,0)=\px u(t,L)=0&\quad t>0,\\
\pxxx u(t,0)=\pxxx u(t,L)=0&\quad t>0,\\
u(0,x)=v_0(x)&\quad 0<x<L,\\
\pt u(0,x)=v_1(x)&\quad 0<x<L,
\end{cases}
\end{equation}
where $\kappa_1$ is a positive constant representing the flexural stiffness of the beam.
We shall assume that
\begin{equation}
\label{ass:init} 
v_0\in H^2(0,L),\qquad v_1\in L^2(0,L).
\end{equation}


The forcing term $\Phi'$ is  thought to model the adhesive like interaction of the beam with the substrate, allowing the 
debonding after a given  threshold is attained. This behavior is the source of a localized nonlinearity which confers to the evolution problem a nontrivial peculiarity.  Then we  assume $\Phi$  is the following function

\begin{equation}
\label{eq:Phi}
\Phi(u)=\begin{cases}
\kappa^2_2\,u^2/2 &\qquad \text{if $|u|\le 1$},\\
\kappa^2_2/2 &\qquad \text{if $|u|> 1$}.
\end{cases}
\end{equation}
In particular we have for all $u\neq \pm 1$
\begin{equation}
\label{eq:Phi'}
\Phi'(u)=\begin{cases}
\kappa^2_2\,u &\qquad \text{if $|u|< 1$},\\
0 &\qquad \text{if $|u|>1$}.
\end{cases}
\end{equation}

The natural energy associated to  \eqref{eq:equation} (i.e.\ to any solution $u$ to \eqref{eq:equation}), is given at time $t$ by the quantity
\begin{equation}
\label{en}
E[u](t)=\int_0^L\left(\frac{(\pt u(t,x))^2+\kappa^2_1(\pxx u(t,x))^2}{2}+\Phi(u(t,x))\right)dx.
\end{equation}

\begin{definition}
\label{def:sol}
We say that a function $u:[0,\infty)\times[0,L]\to\R$ is a dissipative solution of \eqref{eq:equation} if
\begin{itemize}
\item[($i$)] $u\in C([0,\infty)\times[0,L])$;
\item[($ii$)] $\pt u,\,\pxx u\in L^\infty(0,\infty;L^2(0,L))$ and  $u\in L^\infty(0,T;H^2(0,L))$ for every $T>0$;
\item[($iii$)] $\px u(t,0)=\px u(t,L)=0$ for almost every $t>0$;
\item[($iv$)] $u$ is a \em{weak solution} to \eqref{eq:equation}, i.e.\ for every test function $\vfi\in C^\infty(\R^2)$ with compact support
such that $\px\vfi(\cdot,0)=\px \vfi(\cdot,L)=0$
\begin{equation}
\label{eq:weak}
\begin{split}
\int_0^\infty\int_0^L& \left(u\ptt \vfi+\kappa^2_1\pxx u\pxx \vfi+h_u\vfi\right)dtdx\\
&-\int_0^L v_1(x)\vfi(0,x)dx+\int_0^L v_0(x)\pt\vfi(0,x)dx=0,
\end{split}
\end{equation}
where $h_u\in\partial \Phi'\left(u\right)$, that is the subdifferential of $\Phi'(u)$;
\item[($v$)] $u$ may \em{dissipate energy}, i.e.\ for almost every $t>0$: $E[u](t)\leq E[u](0)$, i.e.\ by taking into account \eqref{en},
\begin{equation}
\label{eq:energydissip}
\begin{split}
\int_0^L&\left(\frac{(\pt u(t,x))^2+\kappa^2_1(\pxx u(t,x))^2}{2}+\Phi(u(t,x))\right)dx\\
&\qquad\qquad\le \int_0^L\left(\frac{(v_1(x))^2+\kappa^2_1(\pxx v_0(x))^2}{2}+\Phi(u_0(x))\right)dx.
\end{split}
\end{equation}
\end{itemize}
\end{definition}
In the paper \cite{CDM} the well-posedness of  \eqref{eq:equation} is studied in detail (in the same spirit of the work \cite{CFLM} concerning the adhesive interaction for the second order wave operator) and existence of solutions
in the sense of Definition \ref{def:sol} is proved, moreover some counterxamples to uniqueness of solutions are provided.

The main result of this paper is the following.

\begin{theorem}
\label{th:main}
Let $u$ be a solution of  \eqref{eq:equation} in the sense of Definition \ref{def:sol}.
If
\begin{align}
\label{ass:i} &\text{$v_0$ and $v_1$ are constants with $|v_0|<1$};\\
\label{ass:ii} &\text{there exists a time $\bar{t}>0$ such that $|u(\bar{t},x)|>1$ for every $x\in[0,L]$};\\
\label{ass:iii} &\text{$u\in C^1([0,\infty)\times[0,L])$};\\
\label{ass:iv} &\text{$u$ is energy preserving, i.e. $E[u](t)= E[u](0)$};
\end{align}
then 
\begin{equation*}
\text{$u({t},\cdot)$ is constant for every $t\ge\bar{t}$}.
\end{equation*}
\end{theorem}

\begin{remark}
 \label{rm:1}
 The interesting consequence of the previous theorem relies in enlightening the effect of the nonlinearity hidden in the transition between the two regimes ruled by the conditions $\vert u\vert<1$ and $\vert u\vert>1$ corresponding to attached and detached states. 
 It allows to state that a solution experiencing the transition in a region strictly contained in $(0,L)$ fails to be in 
 $C^1([0,\infty)\times[0,L])$.  
 \end{remark}

\section{Spectral decomposition}
\label{sec:2}

We consider the set of eigenvalues $\{\mu_n\}_{n\in\N}$ and eigenvectors $\{u_n\}_{n\in\N}$ of the 
differential operator $\pxx$ with homogeneous Neumann boundary conditions such that
\begin{equation*}
\begin{cases}
\pxx u_n= \mu_n u_n,\>\>\text{in}\> (0,L),&{}\\
\px u_n(0)=\px u_n(L)=0,&{}
\end{cases}
\qquad \text{and}\qquad\int_{0}^{L}\,u_n u_m\,dx=\delta_{nm}.
\end{equation*}
Since $\pxx$ with homogeneous Neumann boundary conditions is a negative operator (see \cite{B}), the eigenvalues are nonpositive, so we set
$\mu_n=-\lambda^2_n$
and get the identity
\begin{equation}
\pxx u_{n}=-\lambda^2_n\,u_n,\qquad \text{in}\>(0,L).
\label{eq:operator}
\end{equation}
A simple bootstrap argument allows us to deduce that for every $n\in \N$,
\begin{equation}
u_n\in C^\infty([0,L])
\end{equation}
and they are eigenvectors of the differential operator $\pxxxx$ satisfying the following conditions:
\begin{equation}
\label{eq:sp1}
\begin{cases}
\pxxxx u_n= \lambda^4_n u_n,\>\>\text{in}\> (0,L),&{}\\
\px u_n(0)=\px u_n(L)=0,&{}\\
\pxxx u_n(0)=\pxxx u_n(L)=0,&{}\\
\end{cases}
\qquad \text{and}\qquad\int_{0}^{L}\,u_n u_m\,dx=\delta_{nm}.
\end{equation}
A direct computation delivers
\begin{equation}
\label{eq:eigen}
u_n=\begin{cases}
L^{-1/2},& \text{if $n=0$},\\
\displaystyle\sqrt{\frac{2}{L}}\cos\left( \frac{n\pi}{L}x \right),& \text{if $n>0$},
\end{cases}
\qquad\qquad \text{and}\qquad\qquad \lambda_n=\frac{n\pi}{L}.
\end{equation}

Let $u$ be a solution of \eqref{eq:equation} according to Definition \ref{def:sol}.
Since $\{u_n\}_{n\in \N}$ is an Hilbert basis in $H^2(0,L)$ we are allowed to decompose $u$ as follows
\begin{equation}
\label{eq:dec}
u(t,x)=\sum_{n=0}^\infty\alpha_n(t) u_n(x).
\end{equation}

Let us start our analysis by considering the two different regimes $|u|<1$ and $|u|>1$.

\begin{lemma}
\label{lm:<1}
If $|u|<1,$ then 
\begin{equation}
\label{eq:1}
\alpha_n(t)=A_n\cos\left(\omega_n t+\vfi_n\right),\qquad \omega_n=\left[\kappa^2_1\,\left( \frac{n\pi}{L} \right)^4+\kappa^2_2\right]^{1/2},\qquad
E[u](t)=\sum_{n=0}^{\infty}\frac{1}{2}A^2_n\,\omega^2_n,
\end{equation}
for almost every $t>0$, all $n\in\N$ and the constants $A_n,\,\vfi_n$ are implicitly defined through the initial conditions 
\begin{equation}
\label{eq:init<1}
v_0(x)=\sum_{n=0}^\infty A_n\cos(\vfi_n)u_n(x),\qquad
v_1(x)=-\sum_{n=0}^\infty A_n\omega_n\sin(\vfi_n)u_n(x).
\end{equation}
\end{lemma}

\begin{proof}
By using \eqref{eq:sp1} and \eqref{eq:dec} in \eqref{eq:equation} we get
\begin{equation*}
\begin{cases}
\ddot\alpha_n+\left[\kappa^2_1\,\left( \frac{n\pi}{L} \right)^4+\kappa^2_2\right]\,\alpha_n=0,&{}\\
\alpha_n(0)=A_n\cos(\vfi_n),&{}\\
\dot\alpha_n(0)=-A_n\omega_n\sin(\vfi_n),&{}
\end{cases}
\end{equation*}
which returns the first two equations of  \eqref{eq:1}, while the third one follows from
\begin{align}
E[u](t)=&\sum_{n=0}^{\infty}\frac{1}{2}\left\{ \left[\partial_t\left(A_n\cos\left( \omega_n t+\vfi_n \right)\right)\right]^2+\omega^2_n A^2_n\cos^2\left( \omega_n t+\vfi_n \right) \right\}\nonumber\\
=&\sum_{n=0}^{\infty}\frac{1}{2}\left\{ A^2_n\omega^2_n\sin^2\left( \omega_n t+\vfi_n \right)+\omega^2_n A^2_n\cos^2\left( \omega_n t+\vfi_n \right) \right\}
=\sum_{n=0}^{\infty}\frac{1}{2}A^2_n\,\omega^2_n\,.\nonumber
\end{align}
\end{proof}

\begin{lemma}
\label{lm:>1}
If $|u|>1,$ then 
\begin{align}
\label{eq:3}
\alpha_n(t)&=\begin{cases}B_n\cos\left(\nu_n t+\psi_n\right), &\text{if $n\not=0$},\\
C_0+C_1 t, &\text{if $n=0$},
\end{cases}
\qquad \nu_n=|\kappa_1|\,\left( \frac{n\pi}{L} \right)^2,\\
\label{eq:4}
E[u](t)&=\frac{L}{2}\kappa^2_2+\frac{1}{2}C^2_1+\frac{1}{2}\sum_{n=1}^{\infty}B^2_n\,\nu^2_n,
\end{align}
for almost every $t>0$, all $n\in\N$ and the constants $B_n,\,\psi_n$ are implicitly defined through the initial conditions 
\begin{equation*}
v_0(x)=\sum_{n=1}^\infty B_n\cos(\psi_n)u_n(x)+C_0 u_0(x),\qquad
v_1(x)=-\sum_{n=1}^\infty B_n\omega_n\sin(\psi_n)u_n(x)+C_1 u_0(x).
\end{equation*}
\end{lemma}

\begin{proof}
By using \eqref{eq:sp1} and \eqref{eq:dec} in \eqref{eq:equation} we get
\begin{align*}
&\ddot\alpha_n+\kappa^2_1\,\left( \frac{n\pi}{L} \right)^4\,\alpha_n=0,\\
&\alpha_n(0)=
\begin{cases}
A_n\cos(\vfi_n),&\text{if $n\not=0$},\\
C_0,&\text{if $n=0$},
\end{cases}
\quad
\dot\alpha_n(0)=
\begin{cases}
-A_n\omega_n\sin(\vfi_n),&\text{if $n\not=0$},\\
C_1,&\text{if $n=0$},
\end{cases}
\end{align*}
which returns \eqref{eq:3}, while \eqref{eq:4} follows from
\begin{align*}
E[u](t)=&\frac{L}{2}\kappa^2_2+\frac{1}{2}C^2_1+
\sum_{n=1}^{\infty}\frac{1}{2}\left\{ \left[\partial_t\left(B_n\cos\left( \nu_n t+\psi_n \right)\right)\right]^2+\nu^2_n B^2_n\cos^2\left( \nu_n t+\psi_n \right) \right\}\\
=&\frac{L}{2}\kappa^2_2+\frac{1}{2}C^2_1+\sum_{n=1}^{\infty}\frac{1}{2}\left\{ B^2_n\nu^2_n\sin^2\left( \nu_n t+\psi_n \right)+\nu^2_n B^2_n\cos^2\left( \nu_n t+\psi_n \right) \right\}\\
=&\frac{L}{2}\kappa^2_2+\frac{1}{2}C^2_1+\sum_{n=1}^{\infty}\frac{1}{2}B^2_n\,\nu^2_n\,.
\end{align*}
\end{proof}

\begin{proof}[Proof of Theorem \ref{th:main}]
Due to \eqref{ass:i} we decompose the initial data as in \eqref{eq:init<1}.
Thanks to \eqref{eq:dec}, and \eqref{eq:1}, we know that 
\begin{equation*}
u(0,x)=\alpha_0(0) u_0,\qquad u_0=L^{-1/2},\qquad
E[u](0)=\frac{1}{2}A^2_0\omega^2_0=\frac{1}{2}A^2_0\,\kappa^2_2,\qquad
A_n=0,\>n\not=0.
\end{equation*}
Moreover, as long as $|u(t,\cdot)|<1$, due to \eqref{eq:1}, 
\begin{equation*}
|\alpha_0(t)|\le A_0\,L^{-1/2}.
\end{equation*}
Let us distinguish three possible scenarios.

If 
\begin{equation*}
A_0<L^{1/2},
\end{equation*}
 we have that $|u(t,x)|\le |\alpha_0(t)|\le A_0\,L^{-1/2}< 1$. 
 As a consequence \eqref{ass:ii} is never satisfied.
 
If
\begin{equation*}
A_0\ge L^{1/2},
\end{equation*} 
from \eqref{eq:1} in Lemma \ref{lm:<1}, we have that $|\alpha_0|$ monotonically grows up to the value $|u(\bar{t},x)|\ge 1$, for all $x$ in $\left[ 0,L \right]$. From \ref{lm:>1}, we have that $\partial_t|\alpha_0|(\bar{t})\ge 0$ so, by a translation in time argument, we can deal with the solution evaluated in $t-\bar{t}$. Hence, we are in the hypotheses of Lemma \ref{lm:>1}. From here, we then get that
\begin{equation*}
E[u](\bar{t})=\frac{L}{2}\kappa^2_2+\frac{1}{2}C^2_1+\sum_{n=1}^{\infty}\frac{1}{2}B^2_n\,\nu^2_n\,.
\end{equation*}
Due to \eqref{ass:iv}, $E[u](\bar{t})=E[u](0)=\frac{1}{2}A^2_0\kappa^2_2$.

Let us exploit the condition $|\alpha_0(\bar{t})|=L^{1/2}$. Due to the assumption \eqref{ass:iii}, from the \eqref{eq:1} of Lemma \ref{lm:<1} and \eqref{eq:3} of Lemma \ref{lm:>1}, we have that
\begin{equation}
\label{eq:cont}
A_0\cos\left(|\kappa_2|\bar{t}+\phi_o\right)=C_0+C_1 \bar{t}\,.
\end{equation}
Moreover, we also have that
\begin{equation}
\label{eq:der}
-A_0|\kappa_2| \sin\left(|\kappa_2|\,\bar{t}+\phi_0\right)L^{-1/2}=C_1L^{-1/2}\,.
\end{equation}
By combining \eqref{eq:cont} and \eqref{eq:der}, we get
\begin{equation}
C^2_1=\kappa^2_2\left( A^2_0-L \right)\,.
\end{equation}
Therefore
\begin{equation*}
E[u](\bar{t})-E[u](0)=\sum_{n=1}^{\infty}\frac{1}{2}B^2_n\,\nu^2_n=0\,.
\end{equation*}
This implies that $B_n=0$, for all $n\ge 1$. Therefore, the only allowed solution is the constant one.

Eventually, the last case $A_0\ge L^{1/2}$ and $\dot\alpha_0(0)=0$ is already in the assumptions of Lemma \ref{lm:>1}, the solution is
\begin{equation}
u_0=C_0\,L^{-1/2}=A_0\,L^{-1/2}
\end{equation}
where $C_0=A_0$ is due to the continuity and $C_1=0$ is due to the derivability in $t=0$.
 \end{proof}
 
 \begin{corollary}
 \label {cor:1}
 Assume that \eqref{ass:i} and \eqref{ass:ii} hold. Let $u$ be a dissipative solution of \eqref{eq:equation} according to Definition \ref{def:sol}.
 If \eqref{ass:iii} is not satisfied, then
\begin{equation}
\label{eq:cor1} 
|C_1|\le\abs{\kappa_2}\sqrt{A^2_0-L},\qquad\qquad B_n\lsim n^{-5/2},\,\forall n>1.
\end{equation}
 \end{corollary}

 \begin{proof}
 Due to \eqref{ass:i} and \eqref{ass:ii}, we have that $A_0\ge L^{1/2}$.
 Moreover, from  \eqref{ass:iii}, we have that
\begin{equation}
\label{eq:negE}
\frac{L}{2}\kappa^2_2+\frac{1}{2}C^2_1+\frac{1}{2}\sum_{n=1}^{\infty}B^2_n\,\nu^2_n-\frac{1}{2}A^2_0\,\kappa^2_2=E[u](\bar{t})-E[u](0)\le 0\,,
\end{equation}
which leads to
\begin{equation*}
C^2_1-\kappa^2_2\,\left(A^2_0-L\right)\le-\sum_{n=1}^{\infty}B^2_n\,\nu^2_n\le 0,
\end{equation*}
that is the first of \eqref{eq:cor1}.

On the other hand, from \eqref{eq:negE} and \eqref{eq:3} of Lemma \ref{lm:>1}, we have that
\begin{equation}
\label{eq:cutoff}
0\le \sum_{n=1}^{\infty}B^2_n\,n^4\le \frac{\kappa^2_2\left( A^2_0-L \right)\,L^4}{\pi^4\,\kappa^2_1}<\infty,
\end{equation}
and then the second of \eqref{eq:cor1}.
 \end{proof}
 
 \begin{remark}
 \label{rm:2}
 The previous results lead to the following conclusion.
 A solution of \eqref{eq:equation} experiencing a transition from the attached to the  detached regimes,
 according to the requirements to Corollary \ref{cor:1}, exhibits energy migration through the scales ruled by the second  of \eqref{eq:cor1}. 
 \end{remark}

\bibliographystyle{abbrv}
\bibliography{CFP-ref}

\end{document}